\documentclass[11pt]{article}

\usepackage{fullpage}
\usepackage{amsmath}
\usepackage{amsmath}
\usepackage{amssymb}
\usepackage{amsthm}
\usepackage{amscd}
\usepackage{amsfonts}
\usepackage{dsfont}
\usepackage{fancyhdr}
\usepackage{enumerate}
\usepackage[utf8]{inputenc}
\usepackage[margin=1in]{geometry}
\usepackage{color}
\usepackage{hyperref}
\usepackage{tikz}
\usepackage{float}
\usepackage{multicol}
\usepackage{graphicx}

\newtheorem{theorem}{\bf Theorem}[section]
\newtheorem{lemma}[theorem]{\bf Lemma}

\newcommand{\eps}{\varepsilon}

\newcommand{\surp}{{\rm sp}}

\newcommand{\mc}{{\rm mc}}
\newcommand{\chiv}{{\chi_{\rm vec}}}
\newcommand{\vx}{\mathbf{x}}
\newcommand{\vy}{\mathbf{y}}

\usepackage{tikz}

\usepackage{xcolor}

\begin{document}

\title{Tightness of a MaxCut Lower Bound via Vector Chromatic Number}

\author{
Emanuel Juliano\thanks{
Department of Computer Science, Federal University of Minas Gerais, Brazil.
\texttt{emanuelsilva@dcc.ufmg.br}
}
}

\date{}

\maketitle

\begin{abstract}
Recently, Balla, Janzer, and Sudakov showed a lower bound on the MaxCut in terms of the vector chromatic number, recovering known results on the MaxCut of $H$-free graphs. In this note, we show that their bound is tight, providing a construction that achieves a value arbitrarily close to the optimal constant. This answers a question raised by Elphick. Our construction is a modification of the geometric graph used by Feige and Schechtman to establish the integrality gap for the Goemans--Williamson semidefinite relaxation of the MaxCut.
\end{abstract}

\section{Introduction}

The \emph{MaxCut} of a graph $G = (V, E)$, denoted by $\mc(G)$, is the size of the largest bipartite subgraph of $G$. This parameter has been extensively studied over the past 50 years. From a computational perspective, computing the MaxCut of a graph appeared as one of Karp's original 21 NP-complete problems~\cite{karp2009reducibility}. Moreover, it is the first optimization problem for which semidefinite programming was used to obtain a constant-factor approximation guarantee, established in the landmark work of Goemans and Williamson~\cite{goemans1995improved}.

A standard probabilistic argument shows that any graph $G$ with $m$ edges satisfies $\mc(G) \geq m/2$. This motivates the study of the \emph{surplus} of a graph, defined as $\surp(G) = \mc(G) - m/2$. A classical result of Edwards establishes that $\surp(G) \geq \frac{\sqrt{8m + 1}-1}{8}$ for any graph on $m$ edges, a bound that is tight for complete graphs of odd order. To improve this lower bound for graphs that are ``far from complete'', Erd\H{o}s and Lov\'asz~\cite{Erd79} initiated the study of the surplus in $H$-free graphs in the 1970s. A central conjecture in this area, posed by Alon, Bollob\'as, Krivelevich, and Sudakov~\cite{ABKS}, asserts that for any fixed graph $H$, every $H$-free graph $G$ with $m$ edges has a surplus of at least $m^{3/4+\eps_H}$ for some $\eps_H>0$. Recently, a major breakthrough by Jin, Milojevi\'c, Tomon, and Zhang~\cite{jin2025smalleigenvalueslargecuts} made significant progress toward this conjecture by combining novel linear-algebraic techniques with methods from spectral graph theory and semidefinite programming.

The \emph{vector chromatic number} of a graph $G$, denoted $\chiv(G)$, is the minimum real number $\kappa \geq 2$ for which there exists an assignment of unit vectors $\vx_i$ to each vertex $i \in V$ such that $\langle \vx_i, \vx_j\rangle \leq -(\kappa - 1)^{-1}$ for all edges $ij \in E$. This formulation was introduced by Karger, Motwani, and Sudan~\cite{karger1998approximate} to design a semidefinite programming-based approximation algorithm for the graph coloring problem. The vector chromatic number is also famously known as the \emph{Schrijver theta number} of the complement graph, denoted $\vartheta^-(\overline{G})$, a semidefinite program formulation that first appeared in a paper of McEliece, Rodemich, Rumsey, and Welch~\cite{mceliece1978lovasz} and was independently introduced by Schrijver~\cite{schrijver1979comparison}.

Connecting these concepts, a recent work by Balla, Janzer, and Sudakov~\cite{balla2024maxcut} bounded the surplus of a graph in terms of its vector chromatic number, giving a unified approach
and short proofs that recover known results on the MaxCut of $H$-free graphs. 

\begin{theorem}[{\cite[Theorem 1.7]{balla2024maxcut}}]\label{thm:BJS} Let $G$ be a graph which has $m$ edges. Then
\[
    \text{sp}(G) \geq \frac{m}{\pi (\chi_{\text{vec}}(G)-1)}.
\]
\end{theorem}

Theorem~\ref{thm:BJS} is proved by assigning an optimal vector coloring to the vertices of $G$ and applying the classic random hyperplane rounding technique of Goemans and Williamson~\cite{goemans1995improved}. Specifically, given a unit vector embedding where $\langle \vx_i, \vx_j \rangle \leq -(\chiv(G) - 1)^{-1}$ for all edges $ij \in E$, a cut is formed by choosing a uniform random hyperplane through the origin. The lower bound then follows by computing the expected number of edges crossing this hyperplane.

So far, no graph is known to satisfy the inequality in Theorem~\ref{thm:BJS} with the correct constant factor. For this reason, Elphick (personal communication with I. Balla) questioned whether the constant $\pi$ in the denominator could be improved, since computational evidence suggested that it might be replaced by $3$. In this note, we answer this question in the negative by constructing a family of graphs showing that the constant $\pi$ is optimal. Consequently, our result implies that for certain graphs, a random hyperplane cut is essentially the best possible cut achievable from this vector embedding.

\begin{theorem}\label{thm:main}
    For every $\delta \in (0, 1)$, there is a graph $G$ with $m$ edges for which
    \[
    \text{sp}(G) \leq \frac{m}{(\pi-\delta) (\chiv(G)-1)}.
    \]
\end{theorem}

To prove Theorem~\ref{thm:main}, we construct a geometric graph whose vertex set consists of points on the $(d-1)$-dimensional unit sphere $\mathcal{S}^{d-1} \subset \mathds{R}^d$ and two vertices are adjacent if the angle between them is at least $\theta$, for an appropriate choice of $\theta$. This angular threshold condition provides an immediate upper bound on the vector chromatic number. Moreover, in this graph the size of any hyperplane cut is very close to that of the maximum cut. 

Our construction is directly inspired by a seminal work of Feige and Schechtman~\cite{feige2002optimality}, who established the optimal integrality gap $\alpha_{\text{GW}} = \min_{\theta \in [\pi/2, \pi)}\frac{2 \theta}{\pi(1-\cos \theta)}$ for the Goemans--Williamson semidefinite relaxation of MaxCut. In their construction, vertices are adjacent if the angle between them lies within a small interval around the angle $\theta^*$ which is the solution to the minimization problem of $\alpha_{\text{GW}}$. This restriction was necessary to bound the objective value of the semidefinite program relaxation. Since we are only interested in bounding the surplus relative to the vector chromatic number, we simplify both the graph structure and its subsequent analysis.

\section{Construction}

Let $\mathcal{S}^{d-1}$ denote the unit sphere in $\mathds{R}^d$. We measure distances between points on $\mathcal{S}^{d-1}$ by the angle between their respective unit vectors; that is, $d(\vx, \vy) = \arccos(\vx^\top \vy)$. The parameters of our construction are chosen as follows. Fix $\theta \in (\pi/2, \pi)$ and let $\eps \in (0, 1)$. The dimension $d$ is chosen as the smallest integer such that the $(d-1)$-dimensional volume of a spherical cap of radius $\pi-\theta-\eps$ is an $\eps$-fraction of the volume of a cap of radius $\pi-\theta$, which is achieved by taking $d \approx \log (1/\eps)/(\eps(\theta-\pi/2))$.
    
\subsection{Continuous version}

We define a continuous graph $G_c = G_c(\theta, \eps)$ as follows. The vertices of $G_c$ are all points of $\mathcal{S}^{d-1}$. The edges of $G_c$ are pairs of vertices at distance at least $\theta$. Denote by $\mu$ the normalized $(d-1)$-dimensional natural measure on $\mathcal{S}^{d-1}$ and by $\mu^2$ the induced product measure $\mu \times \mu$ on $\mathcal{S}^{d-1} \times \mathcal{S}^{d-1}$. Given a (measurable) subset $A$ of $\mathcal{S}^{d-1}$ and $\theta$ between $0$ and $\pi$, we also denote
\begin{align*}
    \mu_{\theta}(A) &= \mu^{2}\left(\left\{(\vx, \vy) : \vx \in A, \vy \not \in A, \, d(\vx, \vy) \geq \theta\right\}\right).
\end{align*}
The following lemma states which set $A$ maximizes the measure of edges cut.

\begin{lemma}[{\cite[Corollary 6]{feige2002optimality}}]\label{lem:measure}
   Fix $\theta$ between $0$ and $\pi$. Then the maximum of $\mu_\theta(A)$, where $A$ ranges over all (measurable) subsets of $\mathcal{S}^{d-1}$, is attained for any cap of measure $1/2$ (namely, a hemisphere).
\end{lemma}

Therefore, we just need to estimate the measure of edges cut by a hyperplane. Since each hyperplane cuts the same measure of edges, for edges with endpoints at distance $\theta$ apart, a fraction of $\theta/\pi$ of them will be cut by any hyperplane. The following result can be extracted from~\cite[Section 3.1]{feige2002optimality}, but we give the full proof for completeness.

\begin{lemma}\label{lem:opt_c} Let $G_c = G_c(\theta, \eps)$ be defined as above. Denote by $\text{opt}_c$ the measure of edges cut by a random hyperplane, divided by the measure of edges in $G_c$. Then,
     \[\text{opt}_c \leq \theta/\pi + 2\eps.\] 
\end{lemma}
\begin{proof}
Let $E_1$ be the set of edges of $G_c$ whose endpoints have distance at most $\theta+\eps$, and let $E_2$ be the remaining edges. Since the edges in $E_1$ make angles in the range $[\theta, \theta+\eps]$, the probability that a random hyperplane cuts an edge in $E_1$ is at most $(\theta+\eps)/\pi$. By the choice of $d$, the total measure of $E_2$ is at most an $\eps$-fraction of the total measure of $E(G_c)$ and every edge in $E_{2}$ is cut with probability at most $1$. Hence,
\[
\operatorname{opt}_c
\leq
\frac{\mu^2(E_1)}{\mu^2(E(G_c))} \left(\frac{\theta + \eps}{\pi}\right)
+
\frac{\mu^2(E_2)}{\mu^2(E(G_c))}
\leq
\frac{\theta}{\pi} + 2\eps. \qedhere
\]
\end{proof}

\subsection{Discrete version}\label{sec:construction_discrete}

We choose a sufficiently small parameter $\gamma > 0$ (e.g., $\gamma < \eps/(d|\cot(\theta)|) \approx \eps/d$) such that the volume of a spherical cap of radius $\pi-\theta-\gamma$ is at least a $(1-\eps)$-fraction of the volume of a cap of radius $\pi - \theta + \gamma$. The discrete graph $G = G(\theta, \eps)$ is then constructed by partitioning $\mathcal{S}^{d-1}$ into $n = (O(1)/\gamma)^d$ equal-volume cells of diameter at most $\gamma$, following the approach in~\cite[Lemma~21]{feige2002optimality}. The vertex set of $G$ is formed by selecting an arbitrary point within each cell, and two vertices are connected by an edge if the distance between them is at least $\theta$.

\begin{lemma}\label{lem:chiv_bound}
    For every $\theta \in (\pi/2, \pi)$ and for every $\eps \in (0, 1)$, the graph $G = G(\theta, \eps)$ satisfies
    \[
    \chiv(G) \leq 1 - \frac{1}{\cos(\theta)}.
    \]
\end{lemma}
\begin{proof}
Recall that $\chiv(G)$ is the minimum $\kappa \geq 2$ for which there exists an assignment of unit vectors $\vy_i$ to the vertices $i \in V(G)$ satisfying $\vy_i^\top \vy_j \leq -(\kappa - 1)^{-1}$ for all edges $ij \in E(G)$. By construction, each vertex $i \in V(G)$ corresponds to a point $\vx_i \in \mathcal{S}^{d-1}$, and every edge $ij \in E(G)$ satisfies $\vx_i^\top \vx_j \leq \cos(\theta)$. By considering the identity embedding $\vy_i = \vx_i$ for all $i \in V(G)$ and setting $\kappa = 1 - 1/\cos(\theta)$, we obtain $\cos(\theta) = -(\kappa - 1)^{-1}$. It follows immediately that $\chiv(G) \leq \kappa$.
\end{proof}

\begin{lemma}\label{lem:mc_bound}
    For every $\theta \in (\pi/2, \pi)$ and for every $\eps \in (0, 1)$, the graph $G = G(\theta, \eps)$ satisfies
    \[
    \mc(G) \leq \left(\frac{\theta} {\pi} + 5\eps\right) e(G).
    \]
\end{lemma}
\begin{proof}
    Let $C \subseteq V$ be a set such that $e(C, V \setminus C) = \mc(G)$. Color each vertex of $C$ in red and each vertex in $V \setminus C$ in blue and color each of the $n$ cells of the sphere $\mathcal{S}^{d-1}$ with the same color as the vertex on it (like a disco ball). Let $R$ be the set of red points in $\mathcal{S}^{d-1}$.

Denote by $\mu_\theta$ the measure of all edges in $G_c$ and by $\mu_\theta(A, B)$ the measure of edges between cells $A$ and $B$. We give an upper bound on $\mc(G)/e(G)$ from the the ratio $\mu_\theta(R)/\mu_\theta$.

Let $\mathcal{P}$ be the set of all pairs of cells $(A, B)$. We classify the pairs in $\mathcal{P}$ based on the distances between their points. Let $\mathcal{P}_1$ be the set of \textit{near pairs}, where $d(x, y) < \theta$ for all $x \in A$ and $y \in B$. Let $\mathcal{P}_2$ be the set of \textit{contributing pairs}, where $d(x, y) \geq \theta$ for all $x \in A$ and $y \in B$. Because every pair of points between these cells forms an edge, and each cell has a volume of $1/n$, every pair in $\mathcal{P}_2$ contributes exactly $1/n^2$ to the measure of edges. Finally, let $\mathcal{P}_3$ be the set of \textit{mixed pairs}, which contain some points at distance less than $\theta$ and others at distance at least $\theta$.

A cell participates in mixed pairs only with cells entirely contained in the difference of two caps, one of radius $\pi-\theta+\gamma$ and the other of radius $\pi-\theta - \gamma$. As the volume of this region is an $\eps$-fraction of the volume of a cap of radius $\pi-\theta$, the influence of mixed pairs on the total measure of edges is bounded by
\[
\sum_{(A, B) \in \mathcal{P}_3} \mu_\theta(A, B) \leq \frac{|\mathcal{P}_3|}{n^2} \leq \eps \mu_\theta.
\]

Now we are ready to bound $\mu_\theta$ and $\mu_\theta(R)$. Since the near pairs do not contribute to the measure of edges, we can upper bound $\mu_\theta$ by
\[
\mu_\theta = \sum_{(A, B) \in \mathcal{P}_2} \mu_\theta(A, B) + \sum_{(A, B) \in \mathcal{P}_3} \mu_\theta(A, B) \leq \sum_{(A, B) \in \mathcal{P}_2} \frac{1}{n^2} + \eps \mu_\theta \leq \frac{e(G)}{n^2} + \eps \mu_\theta.
\]
Let $\mathcal{P}'_i \subseteq \mathcal{P}_i$ denote the set of pairs $(A, B) \in \mathcal{P}_i$ such that $A$ and $B$ have different colors. Even if every pair in $\mathcal{P}_3'$ contributes to the cut while having zero measure, we can lower bound $\mu_\theta(R)$ by
\[
\mu_\theta(R) = \sum_{(A, B) \in \mathcal{P}_2'} \mu_\theta(A, B) + \sum_{(A, B) \in \mathcal{P}_3'} \mu_\theta(A, B) \geq \frac{|\mathcal{P}_2'|}{n^2} \geq \frac{\mc(G) - |\mathcal{P}_3'|}{n^2} \geq \frac{\mc(G)}{n^2} - \eps\mu_\theta.
\]
We conclude the result by applying Lemma~\ref{lem:opt_c}:
\[
\frac{\mc(G)}{e(G)} \leq \frac{\mu_\theta(R) + \eps\mu_\theta}{\mu_\theta - \eps\mu_\theta} \leq \frac{\mu_\theta(R)}{\mu_\theta} + 3\eps \leq \text{opt}_c + 3\eps \leq \frac{\theta}{\pi} + 5\eps. \qedhere
\]
\end{proof}

\section{Proof of Theorem~\ref{thm:main}}

Before proving Theorem~\ref{thm:main}, we need the following trigonometric estimate.

\begin{lemma}\label{lem:numerical}
   For every $\delta \in (0,1)$, there exist constants $\theta \in (\pi/2, \pi)$ and $\eps \in (0, 1)$ such that
    \[
    \frac{-\arcsin (\cos(\theta))}{\pi} + \eps \leq \frac{-\cos(\theta)}{\pi-\delta}.
    \]
\end{lemma}
\begin{proof}
    Let $\theta$ be such that $-\cos(\theta) = x = \delta/\pi$ and let $\eps = \delta x / (2\pi^2)$.
    From the Taylor expansion of $\arcsin$, since $x > 0$, it follows that:
    \[
    x \leq \arcsin(x) = x + \frac{x^3}{6} + \dots \leq x + x^3.
    \]
    The result then follows by a straightforward computation. By substituting this bound and dropping the strictly negative terms, we have:
    \[
    \left(1 - \frac{\delta}{\pi}\right)\arcsin (x) + (\pi-\delta)\eps \leq x + x^3 - \frac{\delta}{2\pi} x = x + \frac{\delta^3}{\pi^{3}} - \frac{\delta^2}{2\pi^2} \leq x.
    \]
    Dividing this inequality by $\pi-\delta$ and recalling that $\arcsin(x) = -\arcsin(\cos(\theta))$ gives the result.
\end{proof}

Now, we are ready to show our main result.

\begin{proof}[Proof of Theorem~\ref{thm:main}]Let $\theta$ and $\eps$ be chosen as in Lemma~\ref{lem:numerical}, and let $G = G(\theta, \eps/5)$. Applying Lemma~\ref{lem:mc_bound} followed by the identity $\arccos(x) + \arcsin(x) = \pi/2$, we obtain the initial equality below. The subsequent bounds follow directly from Lemma~\ref{lem:numerical} and Lemma~\ref{lem:chiv_bound}, respectively:\[\frac{\mc(G)}{m} - \frac{1}{2} \leq \frac{\theta}{\pi} - \frac{1}{2} + \eps = \frac{-\arcsin (\cos (\theta))}{\pi} + \eps \leq \frac{-\cos (\theta)}{\pi-\delta}  \leq \frac{1}{(\pi-\delta)(\chiv(G) - 1)}. \qedhere
\]
\end{proof}

\section{Concluding Remarks}

In this note we show that for every $\delta \in (0, 1)$, there exists a graph $G$ with $m$ edges for which
\[
\surp(G) \leq \frac{m}{(\pi-\delta)(\chiv(G) - 1)}.
\]

The Lovász theta function of the complement graph $\vartheta(\overline{G})$, also known as the strict vector chromatic number of $G$, is the minimum real number $\kappa \geq 2$ for which there exists a unit vector $\vx_i$ (in some Euclidean space) for each vertex $i \in V$ such that $\vx_i^\top \vx_j = -(\kappa - 1)^{-1}$ holds whenever $i$ and $j$ are distinct vertices in $G$ and $ij \in E$. From this definition, it follows that $\vartheta(\overline{G}) \geq \chiv(G)$. By combining this inequality with Theorem~\ref{thm:BJS}, Balla, Janzer, and Sudakov deduced the following theorem:

\begin{theorem}[{\cite[Theorem 1.5]{balla2024maxcut}}]\label{thm:BJS_theta}
    Let $G$ be a graph with $m$ edges and let $\vartheta = \vartheta(\overline{G})$ be the Lovász theta function of the complement graph $\overline{G}$. Then,
    \[
    \surp(G) \geq \frac{m}{\pi (\vartheta - 1)}.
    \]
\end{theorem}

A natural open question is whether the graph constructed in Section~\ref{sec:construction_discrete} also shows that the constant $\pi$ in Theorem~\ref{thm:BJS_theta} is tight. This would follow immediately if $\vartheta(\overline{G})$ and $\chiv(G)$ were arbitrarily close for $G(\theta, \eps)$. We expect this to be the case, as graphs in which these two parameters differ are rare in the literature~\cite{schrijver1979comparison, balla2024maxcut}.

\section*{Acknowledgements}
I am grateful to my advisor, Gabriel Coutinho, for his insightful feedback and mentorship. I would also like to thank Igor Balla for the fun and engaging discussions. This project was supported by FAPEMIG and CNPq.


\bibliographystyle{abbrv}
\bibliography{references.bib}

\begin{thebibliography}{10}

\bibitem{ABKS}
N.~Alon, B.~Bollob{\'a}s, M.~Krivelevich, and B.~Sudakov.
\newblock Maximum cuts and judicious partitions in graphs without short cycles.
\newblock {\em Journal of Combinatorial Theory, Series B}, 88(2):329--346, 2003.

\bibitem{balla2024maxcut}
I.~Balla, O.~Janzer, and B.~Sudakov.
\newblock On maxcut and the lov{\'a}sz theta function.
\newblock {\em Proceedings of the American Mathematical Society}, 152(05):1871--1879, 2024.

\bibitem{Erd79}
P.~Erd{\H{o}}s.
\newblock Problems and results in graph theory and combinatorial analysis.
\newblock In {\em Graph Theory and Related Topics (Proc. Conf., Univ. Waterloo, Waterloo, 1977)}, pages 153--163. Academic Press, 1979.

\bibitem{feige2002optimality}
U.~Feige and G.~Schechtman.
\newblock On the optimality of the random hyperplane rounding technique for max cut.
\newblock {\em Random Structures \& Algorithms}, 20(3):403--440, 2002.

\bibitem{goemans1995improved}
M.~X. Goemans and D.~P. Williamson.
\newblock Improved approximation algorithms for maximum cut and satisfiability problems using semidefinite programming.
\newblock {\em Journal of the ACM (JACM)}, 42(6):1115--1145, 1995.

\bibitem{jin2025smalleigenvalueslargecuts}
Z.~Jin, A.~Milojevi{\'c}, I.~Tomon, and S.~Zhang.
\newblock From small eigenvalues to large cuts, and chowla's cosine problem, 2025.
\newblock arXiv:2509.03490.

\bibitem{karger1998approximate}
D.~Karger, R.~Motwani, and M.~Sudan.
\newblock Approximate graph coloring by semidefinite programming.
\newblock {\em Journal of the ACM (JACM)}, 45(2):246--265, 1998.

\bibitem{karp2009reducibility}
R.~M. Karp.
\newblock Reducibility among combinatorial problems.
\newblock In {\em 50 Years of Integer Programming 1958-2008: from the Early Years to the State-of-the-Art}, pages 219--241. Springer, 2009.

\bibitem{mceliece1978lovasz}
R.~J. McEliece, E.~R. Rodemich, and H.~C. Rumsey~Jr.
\newblock The lov{\'a}sz bound and some generalizations.
\newblock {\em J. Combin. Inform. System Sci}, 3(3):134--152, 1978.

\bibitem{schrijver1979comparison}
A.~Schrijver.
\newblock A comparison of the delsarte and lov{\'a}sz bounds.
\newblock {\em IEEE Transactions on Information Theory}, 25(4):425--429, 1979.

\end{thebibliography}

\end{document}